\documentclass[12pt]{article}

\usepackage{amsrefs}
\usepackage{amsmath,amssymb,amsthm,enumerate}
\usepackage[all,cmtip]{xy}
\usepackage[applemac]{inputenc}
\usepackage[colorlinks=true, linkcolor=blue, citecolor=blue]{hyperref}

\def \1{{\bf 1}}
\def \a{\mathfrak a}
\def \A{{\mathbb A}}
\def \al{\alpha}
\def \Aut{\operatorname{Aut}}
\def \b{\mathfrak b}
\def \bs{\backslash}
\def \C{{\mathbb C}}

\def \CC{\mathcal{C}}

\def \CF{\mathcal{F}}
\def \CG{\mathcal{G}}

\def \CO{\mathcal{O}}

\def \eps{\varepsilon}
\def \F{{\mathbb F}}
\def \fin{\mathrm{fin}}
\def \Ga{\Gamma}
\def \gen{\mathrm{gen}}
\def \GL{\operatorname{GL}}
\def \ga{\gamma}

\def \Id{{\rm Id}}

\def \Ind{\operatorname{Ind}}
\def \inert{\mathrm{inert}}
\def \la{\lambda}
\def \La{\Lambda}
\def \M{\operatorname M}
\def \Min{\operatorname{Min}}

\def \N{{\mathbb N}}
\def \ol{\overline}

\def \PGL{\operatorname{PGL}}
\def \R{{\mathbb R}}

\def \SL{\operatorname{SL}}
\def \sm{\smallsetminus}
\def \spn{\operatorname{span}}

\def \vert{\operatorname{vert}}
\def \vol{{\rm vol}}

\def \Z{{\mathbb Z}}
\def \({\left(}
\def \){\right)}

\newcommand{\e}
[1]{\emph{#1}\index{#1}}

\newcommand{\norm}
[1]{\left\|#1\right\|}

\newtheorem{theorem}{Theorem}[section]

\newtheorem{lemma}[theorem]{Lemma}

\theoremstyle{definition}
\newtheorem{definition}[theorem]{Definition}

\begin{document}

\pagestyle{myheadings} \markright{PRIME GEODESIC THEOREM FOR BUILDINGS}

\title{A prime geodesic theorem for higher rank buildings}
\author{Anton Deitmar \& Rupert McCallum}
\date{}
\maketitle

{\bf Abstract:} We prove a prime geodesic theorem for compact quotients of affine buildings and apply it to get class number asymptotics for global fields of positive characteristic.  

$$ $$

\tableofcontents

\newpage
\section*{Introduction}

The prime geodesic theorem for a compact, negatively curved Riemannian manifold $M$ states that the number $\pi(x)$ of prime closed geodesics of length $\le x$ on $M$ satisfies
$$
\pi(x)\sim \frac{e^{Ax}}{Bx},
$$
where $A,B>0$ are (explicit) constants depending on the manifold.
It has been applied in \cite{Sarnak} to arithmetic quotients of symmetric spaces to derive class number asymptotics.
In case of hyperbolic spaces, remainder terms have been given, see for instance \cite{Sound,FJK}.
In \cite{class} this has been extended to a case of higher rank and in \cite{classNC} to a non-compact situation.
The full higher rank case has been explored in \cite{HR,HR2}.

The $p$-adic counterpart of symmetric spaces are Bruhat-Tits buildings.
In the rank one case, these are trees.
In \cite{Kyunshang} we applied the Prime Geodesic Theorem for graphs to get class number asymptotics for orders over imaginary quadratic fields.
See \cite{Raulf} for a different approach to this case.
In the current paper we derive the Prime Geodesic Theorem for quotients of buildings and we use it for class number asymptotics for global fields of characteristic zero.
The result is an equidistribution assertion for units in orders over a given global field, where the units are weighted by class numbers and regulators.

\section{The building}
In this section we recall some results of \cite{BL}.
Let $X$ be a locally finite, simplicial, thick, affine building of dimension $d=\dim X$.
Recall that a \e{type function} on a chamber is a bijection from the set $\vert(C)$ of vertices of $C$ to the set $\{0,1,\dots,d\}$.
Further, a \e{type function} on $X$ is a map from the set $\vert(X)$ of vertices of $X$ to $\{0,\dots,d\}$ that restricts to a type function on each chamber.

Given a chamber $C$ the restriction defines a bijection between the set of all type functions on $X$ and the set of all type functions on $C$.
An automorphism $g$ of $X$ is said to be \e{type preserving}, if $g$ preserves one (and hence any) type on $X$.

Let $\Aut(X)$ denote the group of automorphisms of the building $X$, i.e., the group of automorphisms of the complex mapping apartments to apartments.
The group $\Aut(X)$ is a totally disconnected locally compact group where a basis of unit-neighborhoods is given by the set of all pointwise stabilizers of finite sets in $X$.
Let 
$$
G\subset\Aut(X)
$$ 
be an open, finite index subgroup. 
Note that any open subgroup is closed and that any finite index closed subgroup is open.
Let $G^0\subset G$ be the subgroup of all type preserving automorphisms in $G$.
It is normal in $G$ and $G/G^0$ is a finite group.

We assume that $G^0$ acts \e{strongly transitively}, i.e., $G^0$ acts transitively on the set of all pairs $(\a,C)$, where $\a$ is an apartment of $X$ and $C$ is a chamber in $\a$. 
See \cite{AB}, Chapter 6.
This is equivalent to saying that $G^0$ acts chamber-transively on $X$ and the pointwise stabilizer $K_C$ of a given champber $C$ acts transitively on the set of all apartments $\a$ containing $C$.

Let $\partial X$ be the visibility boundary of $X$ and let $\CC\subset\partial X$ be a spherical chamber and fix an apartment $\a\subset X$ such that $\CC\subset\partial \a$. 
Let $P=P_\CC$ be the point-wise stabilizer of $\CC$, then  $P$ is called a \e{minimal parabolic} subgroup of $G$.
In the boundary $\partial\a$, there is a unique chamber $\ol\CC$ opposite to $\CC$.
Let $\ol P$ denite the corresponding parabolic and set
$$
L=P\cap\ol P.
$$
We call $L$ a \e{Levi-component} of $P$.
Let $M$ denote the point-wise stabilizer of $\a$, then $M$ is normal in $L$ and the quotient $A=L/M$ is free abelian of rank $d$. (Note that in \cite{BL} we wrote ``rank $\le d$'', this discrepancy is due to the fact that we are now assuming $G^0$ to act strongly transitively, which we did not assume in \cite{BL}.)
The group $A$ acts on $\a$ by translations. 
Let $A^-$ denote the open cone in $A$ of all elements translating towards a point in the open spherical chamber $\ol\CC$.

Fix a fundamental chamber $C_0$ in $\a$ and let $K\subset G$ denote its \e{pointwise} stabilizer.
As $M\subset K$ it makes sense to write $KAK$ for the set $KLK$.

\begin{definition}\label{defgeneric}
Since $X$ is affine, there exists a metric $d$ on $X$ which is euclidean on each apartment and $G$-invariant.
For $g\in G$ let 
$$
d(g)=\inf_{x\in X}d(gx,x).
$$
As $G$ acts cellularly, this infimum is always attained, hence a minimum.
An element $g\in G$ is called \e{hyperbolic}, if $d(g)>0$, the minimum is attained on a convex subset which is the union of parallel lines along which $g$ is a translation \cite{BL}.
This set is called the \e{minimal set} of $g$ and is written as $\Min(g)$.
We call $g$  a \e{generic} element, if the so defined element of $\partial X$ is generic, i.e., is not in a lower dimensional stratum.
This implies that the minimal set $\Min(g)$ consists of exactly one apartment.
\end{definition}

\section{Closed Geodesics}
\begin{definition}
A \e{geodesic curve} in $X$ is a curve $c:\R\to X$, which in the euclidean structure of each apartment is a straight line which is parametrized at unit speed.
Two curves $c,c'$ are called equivalent if there exists $t_0\in\R$ such that $c'(t)=c(t+t_0)$ holds for all $t\in\R$.
A \e{geodesic} in $X$ is an equivalence class of geodesic curves.
\end{definition}

\begin{definition}
Let $\Gamma\subset G$ be a discrete,  cocompact subgroup of $G$.
By a \e{geodesic in $\Ga\bs X$} we understand the image of a geodesic in $X$.
This is the proper definition in the case when $\Ga$ is allowed to have torsion elements and thus the local euclidean structure is not necessarily preserved in the quotient $\Ga\bs X$.
A geodesic $c$ in $\Ga\bs X$ is called a \e{closed geodesic}, if $c(t+l)=c(t)$ holds for some $l>0$ and all $t\in\R$.
A closed geodesic $c$ in in $\Ga\bs X$ thus lifts to a geodesic $\tilde c$ in $X$.
The forward direction of $\tilde c$ defines a point $b\in\partial X$.
We call $c$ a \e{generic} geodesic, if $b$ is a generic point of the spherical building $\partial X$, i.e., if $b$ is not contained in any wall of a Weyl chamber.
\end{definition}

For a given closed geodesic $c$ there exists $\ga\in\Ga$ closing it, i.e., if $c(t+l)=c(t)$ for all $t\in\R$ and some $l>0$, then $\tilde c(t+l)=\ga\tilde c(t)$ holds for a uniquely determined $\ga\in\Ga$ and all $t\in\R$.
This element $\ga$ is hyperbolic and $\tilde c$ lies in its minimal set $\Min(\ga)$.

\begin{lemma}\label{lem2.1}
Let $A^0=A\cap G^0$, or rather $A^0=L^0/M$, where $L^0=L\cap G^0$.
\begin{enumerate}[\rm (a)]
\item 
The group $G^0$ is generated by the set $KA^0K$. Likewise, it is generated by the union of all pointwise stabilizers $K_C$ of chambers $C$ in $X$.
\item If $\ga\in\Ga$ closes a generic geodesic, then $\ga$ is hyperbolic and conjugate in $G$ to an element of $KAK$.
In particular, $\ga$ is generic (Definition \ref{defgeneric}).
\item Let $\theta: G^0\to H$ be a group homomorphism with $\theta(K)=1$. Then $\theta\equiv 1$.
\end{enumerate}
\end{lemma}

\begin{proof}
(a) 
Let $C_0$ be the fundamental chamber and let $g\in G^0$.
If $gC_0=C_0$ then $g$ lies in $K$ already and we are done.
Otherwise, we have to show that the group $H$ generated by $KAK$ contains an element $h$ such that $gC_0=hC_0$, because then $h^{-1}g\in K$ and (c) is proven.
So we have to show that $HC_0$ contains all chambers of $X$.

We start by showing that it contains all direct neighbors $C$ of $C_0$.
Choose a labeling $v_0,\dots,v_d$ of the vertices of $C_0$.
Let $C$ be a neighbor sharing the face $F$ with vertices $v_1,\dots,v_d$ with $C_0$.
As $X$ is thick, there is another such neighbor $D$ of $C_0$. Let $\a$ be an apartment containing $C_0$ and $D$. Let $s$ be the reflection at $F$ in $\a$.
Let $T$ be the Weyl translation with $Tv_0=w_0=s(v_0)$ and let $E$ be the chamber $T(C_0)$.
The translation $T$ extends to an element of $KAK\subset H$, so that the group $TKT^{-1}$ lies in $H$.
This group, however, is the pointwise stabilizer $K_E$ of the chamber $E$. 
Let now $\a'$ be an apartment containing $E$ and $C$, then there is $k\in K_E\subset H$ with $k\a=\a'$, hence $kC_0=C$ as claimed.
This proof can now be iterated to show that $HC_0$ contains the direct neighbors of direct neighbors of $C_0$ and so on. 
Examining the proof, one finds that we have shown that the group generated by all stabilizers of chambers generates $G$.

(b) If $\ga$ closes a generic geodesic, it must be hyperbolic and its minimal set is an apartment $\a'$, \cite{BL}. So $\ga$ induces a translation on this apartment, which makes it an element of $K'A'K$ where $K'$ is the stabilizer of a chamber in $\a'$. Now the set $K'A'K'$ is $G$-conjugate to $KAK$ and the claim follows.

(c) If $\theta(K)=1$, then $\theta(gKg^{-1})=1$ for every $g\in G^0$, but these groups generate $G^0$ byt part (a), so $\theta\equiv 1$.
\end{proof}

\section{The zeta function}
Let $\ga\in\Ga$ be a generic element as in Definition \ref{defgeneric}.
The minimal set $\Min(\ga)$ is an apartment $\a$.
Let $G_\a$ denote the stabilizer of $\a$ in $G$ and let $G_\ga$ and $\Ga_\ga$ denote the respective centralizers of $\ga$.
Let $G^\a$ denote the image of $G_\a$ in $\Aut(\a)$ and let $\Ga_\ga^\a\subset G_\ga^\a\subset G^\a$ denote the image of the groups $\Ga_\ga$ and $G_\ga$ in $G^\a$.
Then the set $G^\a_\ga/\Ga_\ga^\a$ is finite. 
Note that, as $G^0$ acts strongly transitively, $G^\a$ contains the Weyl group of $\a$.

For a measurable set $M\subset G$ we set
$$
[\ga:M]=\vol\(\{x\in G/G_\ga: x\ga x^{-1}\in M\}\).
$$
In \cite{BL} it is shown that the index $[\ga:KA^-K]$ is a natural number.

\begin{lemma}
For an element $\ga\in G$ the following are equivalent:
\begin{enumerate}[\rm (a)]
\item $\ga$ is generic,
\item $\ga$ is conjugate to an element of $KaK$ with $a\in A^-$. 
\end{enumerate}
In this case the element $a$ is uniquely determined by $\ga$ and the index satisfies $[\ga:KA^-K]=[\ga:KaK]=1$. 
\end{lemma}

\begin{proof}
If $\ga$ is generic, then, as $G^0$ acts strongly transitively, modulo conjugation we can assume $\Min(\ga)=\a$ and thus $\ga=ak$ for some generic $a\in A$ and $k\in K$. 
As the image of $G^\a$ in $\Aut(\a)$ contains the Weyl group, we can conjugate $a$ into $A^-$. This establishes (a)$\Rightarrow$(b).

The assertion (b)$\Rightarrow$(a) and the uniqueness of $a$ is valid without the assumption of strong transitivity and is proven in Lemma 2.3.8 of \cite{BL}.
The assertion about the index rests on strong transitivity.
With the given normalizations,
\begin{align*}
[\ga:KaK]&=\vol\(\{x\in G/G_\ga: x\ga x^{-1}\in KaK\}\)\\
&=\left|\big\{a\in K\bs G/G_\ga:x\ga x^{-1}\in KaK\big\}\right|.
\end{align*}
So we need to show that if $\ga,x\ga x^{-1}$ both lie in $KaK$, then $x\in KG_\ga$.
If $\ga,x\ga x^{-1}\in KaK$, we can replace both with $K$-conjugates to get them into $aK$.
This means that $\ga C_0=x\ga x^{-1}C_0=aC_0$, so $C_0,aC_0$ lie in $\Min(\ga)$ and in $\Min(x\ga x^{-1})=x\Min(\ga)$.
That means that $\b=\Min(\ga)$ and $x\b$ are apartments containing $C_0$.
By strong transitivity, $K$ acts transitively on these apartments, therefore there are $k_1,k_2\in K$ with
$$
\a=k_1\Min(\ga)=\Min(k_2\ga k_1^{-1})=k_2x\Min(\ga)=\Min(k_2x\ga x^{-1}k_2^{-1}).
$$
Replacing $\ga$ with $k_1\ga k_1^{-1}$ and $x$ with $k_2x$ we can assume that $\ga$ and $x\ga x^{-1}$ preserve $\a$ and both act as the same translation $y\mapsto ay$ on $\a$.
Then $\a=\Min(x\ga x^{-1})=x\Min(\ga)=x\a$, so $x$ preserves $\a$ as well. As $x\ga x^{-1}$ and $\ga$ act as the same generic translation on $\a$, $x$ itself acts by a translation on $\a$.
There is an element $y\in G_\ga$ acting by the same translation as $x$, so $xy^{-1}\in K$. The claim follows.
\end{proof}

\begin{definition}\label{def3.1}
Let $\Ga^\gen$ denote the set of generic elements in $\Ga$ and let $[\Ga^\gen]$ denote the set of $\Ga$-conjugacy classes in $\Ga^\gen$.
Let $\a$ denote the apartment used to define $A$. Let $v_0$ be a special vertex in $\a$ and let $C$ be the unique chamber in $\a$ with vertex $v_0$ such that the wall $W$ of $C$ which is opposite to $v_0$, faces $\ol\CC$.
Let $v_1,\dots,v_d$ denote the remaining vertices of $C$.
The map $v_j\mapsto j$ extends in a unique way to a map from the set $V(X)$ of vertices of $X$ to $\{0,1,\dots,d\}$ which is injective on the set $V(D)$ of vertices of any given chamber $D$.
The image of a vertex is called the \e{type} of the vertex.
Then all vertices of type zero are special vertices.

Using $v_0$ as origin we give $\a$ the structure of a real vector spaces and $v_1,\dots,v_d$ is a basis.
Let $e_j=r_jv_j$, where $r_j>0$ is the smallest rational number such that all vertices of type zero are contained in
$$
Z=\Z e_1\oplus\dots\oplus\Z e_d.
$$
A given $a\in A$ acts on $\a$ by translation $ax=x+v_a$ where
$$
v_a=\la_1(a)e_1+\dots+\la_d(a)e_d
$$
is the translation vector. Since this translation respects the simplicial structure, the numbers $\la_1(a),\dots,\la_d(a)$ are integers.
Indeed, the map
\begin{align*}
\la:A&\mapsto \Z^d,\\
a&\mapsto (\la_1(a),\dots,\la_d(a))
\end{align*}
is an isomorphism of the group $A$ to a lattice $\La=\La_A\subset\Z^d$, which maps the cone $A^-$ to the cone 
$$
\La^+=\big\{\la\in\La:\la_1,\dots,\la_d>0\big\}.
$$

\end{definition}
For $u\in\C^d$ and $a\in A^-$ we write
$$
u^a=u_1^{\la_1(a)}\cdots u_d^{\la_d(a)},
$$
and define
$$
S(u)=\sum_{[\ga]\in[\Ga^\gen]}|G_\ga^\a/\Ga_\ga^\a| \ u^{a_\ga},
$$
where $a_\ga\in A^-$ is the unique element such that $\ga$ is $G$-conjugate to an element of $Ka_\ga K$.
Theorem 2.4.2 of \cite{BL} states that the series $S(u)$ converges for small $u$ to a rational function. More precisely, there exists a finite set $E\subset A$, elements $a_1,\dots,a_d\in A^-$ and quasi-characters $\eta_1,\dots,\eta_r:A\to\C^\times\cup{0}$ such that
$$
S(u)=\sum_{j=1}^r\sum_{e\in E}\frac{\eta_j(e)u^e}{(1-\eta_j(a_1)u^{a_1})\cdots(1-\eta_j(a_d)u^{a_d})}.
$$
Moreover, the space $L(\Ga\bs G)^K\cong L^2(\Ga\bs G/K)$ has a basis $\phi_1,\dots,\phi_r$ such that all $R(\1_{KaK})$ with $a\in A^-$ are in Jordan normal form with respect to this basis.
In particular,
$$
R(\1_{KaK})\phi_j-\eta_j(a)\phi_j\in\spn(\phi_1,\dots,\phi_{j-1})
$$
holds for every $j$ and all $a\in A^-$ and this equation defines the quasi-character $\eta_j$.
Here $R$ is the right translation representation of $G$ on $L^2(\Ga\bs G)$, so $R(y)\phi(x)=\phi(xy)$, $x,y\in G$ and for a function $f$ on $G$ (like $f=\1_{KaK}$) we define $R(f)$ by integration:
$$
R(f)\phi(x)=\int_G f(y)\phi(xy)\,dy.
$$
Note that the space $L^2(\Ga\bs G)^K$ contains the constant function $\phi\equiv 1$. Then
$$
R(\1_{KaK})\phi(x)=\int_G\1_{KaK}(y)\phi(xy)\,dy=\vol(KaK)=|KaK/K|,
$$
since we normalize the Haar measure by giving the compact open subgroup $K$ volume 1. 
We can assume $\phi_1\equiv 1$.

At this point we note that we have a certain amount of freedom in choosing the group $G\subset\Aut(X)$.
By changing $G$ if necessary, we can assume that $G$ is generated by $G^0$ together with $\Ga$.
Since $\Ga$ and $G^0$ are subgroups and $G^0$ is a normal subgroup, this means that we have
$$
G=\Ga G^0,
$$
i.e., every $g\in G$ can be written as a product $g=\ga g^0$ of some $\ga\in\Ga$ and an element $g^0$ of $G^0$.

\begin{lemma}\label{lem3.2}
For $j=1,\dots,r$ we have
$$
|\eta_j(a)|\le |KaK/K|,\quad a\in A^-.
$$
For every $j\ge 2$ there exists $a\in A^-$ with
$$
|\eta_j(a)|< |KaK/K|.
$$
\end{lemma}

\begin{proof}
For $1\le j\le r$ we have
\begin{align*}
\left| R(\1_{KaK})\phi(x)\right|
&\le \int_{KaK}|\phi(xy)|\,dy\le |KaK/K|\ \norm\phi_{\Ga\bs G},
\end{align*}
where $\norm\phi_{\Ga\bs G}=\sup_{x\in\Ga\bs G}|\phi(x)|$.
Suppose that $R(\1_{KaK})\phi=\eta(a)\phi$ holds for every $a\in A^-$, then, taking supremum over $x$ yields
$$
|\eta(a)|\ \norm\phi_{\Ga\bs G}\le|KaK/K|\ \norm\phi_{\Ga\bs G}.
$$
For the second assertion assume additionally $|\eta(a)|=|KaK/K|=\eta_1(a)$ for every $a\in A^-$. We have to show that $\phi$ is constant.
Let $a\in A^-$ and write $KaK=\bigsqcup_{j=1}^s k_jaK$, then $s=|KaK/K|=\eta_1(a)$ and
$$
R(\1_{KaK})\phi=\sum_{j=1}^{s}R(k_j)R(a)\phi.
$$
For the $L^2$-norm we have, since $R(k_j)R(a)$ is unitary,
\begin{align*}
s\norm\phi_2&=|\eta(a)|\norm\phi_2
=\norm{\sum_{j=1}^{s}R(k_j)R(a)\phi}_2\\
&\le\sum_{j=1}^{s}\norm\phi_2=s\norm\phi_2.
\end{align*}
We get equality everywhere, but this can only happen if the $R(k_j)R(a)\phi$ are all $\R_+$-colinear and by unitarity this means that $R(k_j)R(a)\phi=R(a)\phi$ for every $j$.
Since $s=\eta_1(a)$, follows
$$
R(a)\phi=\frac{\eta(a)}{\eta_1(a)}\phi.
$$
As $\phi$ is $K$-invariant, we have for $k_1,k_2\in K$,
$$
R(k_1 ak_2)\phi=R(k_1)R(a)R(k_2)\phi=R(k_1)R(a)\phi=\frac{\eta(a)}{\eta_1(a)}R(k_1)\phi= \frac{\eta(a)}{\eta_1(a)}\phi.
$$
Since $R$ is a representation, the map $\theta=\frac\eta{\eta_1}$ therefore extends to a character on the group generated by $KA^-K$, which contains $G^0$.
By Lemma \ref{lem2.1} we get $\theta(G^0)=1$, so for every $g^0\in G^0$ we have 
$$
\phi(g^0)=R(g_0)\phi(1)=\theta(g_0)\phi(1)=\phi(1).
$$
Finally, as $G=\Ga G^0$ we write any given $g\in G$ as $g=\ga g^0$ accordingly and we get
\begin{align*}
\phi(g)&=\phi(\ga g^0)=\phi(g^0)=\phi(1).
\end{align*}
So $\phi$ is constant, as claimed.
\end{proof}

\section{The prime geodesic theorem}
\begin{definition}
For $\ga\in\Ga^\gen$ we write $\Ind(\ga)=|G_\ga^\a/\Ga_\ga^\a|$, where $\a=\Min(\ga)$.
For $k\in\N^d$ let 
$$
N(k)=\sum_{[\ga]:\la(a_\ga)=k}\Ind(\ga),
$$
where the sum runs over all conjugacy classes $[\ga]$ in $\Ga^\gen$ such that $\la(a_\ga)=k$.
So $N(k)=0$ if $k\notin\La$, the lattice of Definition \ref{def3.1}.
\end{definition}

In the following, for $k\in\N^d$ and $c\in\R^d$ we write
$$
c^k=c_1^{k_1}\cdots c_d^{k_d}.
$$

\begin{theorem}
[Prime Geodesic Theorem]
Let $\La\subset\Z^d$ be the lattice of Definition \ref{def3.1}.
There exists a sublattice $\La'\subset \La$ and a function $C_{\La/\La'}:\La/\La'\to (0,\infty)$, and constants $c_1,\dots,c_d>1$ such that
for $k_j\to\infty$ independently, we have 
$$
N(k)\quad\sim\quad\1_\La(k)\ C_{\La/\La'}(k)\ c^k.
$$
Explicit formulae for the constants and the function $C_{\La/\La'}$ are given below.
\end{theorem}

\begin{proof}
We have
\begin{align*}
S(u)&=\sum_{k\in\N^d}N(k)u^k\\
&= \sum_{j=1}^r\sum_{e\in E}\sum_{m\in\N_0^d}\eta_j(e)\eta_j(a_1)^{m_1}\cdots\eta_j(a_d)^{m_d} u^{\la(m\cdot a+e)},
\end{align*}
where $m\cdot a$ stands for $m_1a_1+\dots+m_d a_d$.
This implies
$$
N(k)=\sum_{j=1}^r\underbrace{\sum_{\substack{e\in E\\ m\in\N^d\\ \la(a\cdot m+e)=k}}\eta_j(e)\eta_j(a_1)^{m_1}\cdots\eta_j(a_d)^{m_d}}_{=N_j(k)}.
$$
We shall show that the term $N_1(k)$ dominates as $k\to\infty$.
To simplify the notation  we write $\eta(x)=\eta_1(x)=|KxK/K|$.
The proof of Lemma 2.4.4 in \cite{BL} reveals that each $a_j$ is a multiple of $v_j$ and hence the equation $\la(a\cdot m+e)=k$ implies $\la_\nu(a_\nu)m_\nu+\la_\nu(e)=k_\nu$, $\nu=1,\dots,d$, so that
$$
m_\nu=\frac{k_\nu-\la_\nu(e)}{\la_\nu(a_\nu)}.
$$
Let $\La'$ be the sublattice of $\La$ generated by $\la(a_1),\dots,\la(a_d)$.
For each $k\in\N^d$ there is at most one  $e\in E$ with $\la(a\cdot m+e)=k$.
This element $e$ only depends on $k$ up to $\La'$.
We get
$$
N_1(k)=\eta(a_1)^{\frac{k_1}{\la_1(a_1)}}
\cdots \eta(a_d)^{\frac{k_d}{\la_1(a_d)}}
\sum_{\substack{e\in E\\ m\in\N^d\\ \la(a\cdot m+e)=k}} \eta(e)\eta(a_1)^{\frac{\la_1(e)}{\la_1(a_1)}}
\cdots
\eta(a_d)^{\frac{\la_d(e)}{\la_d(a_d)}}.
$$
The sum
$$
C_{\La/\La'}(k)=\sum_{\substack{e\in E\\ m\in\N^d\\ \la(a\cdot m+e)=k}} \eta(e)\eta(a_1)^{\frac{\la_1(e)}{\la_1(a_1)}}
\cdots
\eta(a_d)^{\frac{\la_d(e)}{\la_d(a_d)}}
$$
has at most one summand and depends on $k$ only up to $\La'$ and is non-zero if and only if $k\in\La$.
Setting $c_\nu=\eta(a_\nu)^{\frac1{\la_\nu(a_\nu)}}$ we get the desired asymptotic for $N_1(k)$ instead of $N(k)$.

By Lemma \ref{lem3.2}, for each $j=2,\dots,r$ there exists $\nu(j)$ such that $|\eta_j(a_{\nu(j)})|<\eta(a_{\nu(j)})$.
Set
$$
\theta=\max_{j\ge 2}\frac{|\eta_j(a_{\nu(j)})|}{\eta(a_{\nu(j)})}<1.
$$
It then follows that
$$
|N_2(k)+\dots+N_r(k)|\le\(\sum_{j=2}^r\theta^{k_{\nu(j)}}\)N_1(k).
$$
This implies that $N_1(k)$ dominates the asymptotic and thus we get the claim for $N(k)$.
\end{proof}

\section{Division algebras}
Let $R$ be an integral domain and $K$ its field of fractions.
Let $A$ be a finite-dimensional $K$-algebra with unit.
An \e{$R$-order} in $A$ is an $R$-subalgebra $\Lambda$ of $A$, which is finitely generated as $R$-module and spans the $K$-vector space $A$, i.e., $K\La=A$.

Now assume that $K$ is a global field of positive characteristic and $R\subset K$ is a Dedekind domain with fraction field $K$.
Then $K$ is the function field of a curve $\CC$ over a finite field $k$. An example of a possible ring $R$ would be the coordinate ring of the affine curve $\CC\sm\{\infty\}$, where $\infty\in\CC$ is a rational closed point

If the $K$-algebra $A$ is a global field $F$ over $K$ and $\CO$ is an $R$-order in $F$, let $I(\CO)$ denote the set of all finitely generated $\CO$-submodules of $F$.
By the Jordan-Zassenhaus Theorem \cite{Reiner}, the set $[I(\CO)]$ of isomorphism classes of elements of $I(\CO)$ is finite.
Let $h(\CO)$ be its cardinality, called the \e{class number} of the order $\CO$.

If $F$ is a global field over $K$, there is a maximal $R$-order $\CO_F$, which is the integral closure of $R$ in $F$, and wich contains any other order in $F$.
The same applies in the local situation, if $F$ is a finite extension of $K_v$ for a place $v$, the integral closure $\CO_F$ of $R_v$ is a maximal $R_v$-order containing every other order.
We say that an $R$-order $\CO$ of $F$ is \e{maximal at $v$}, if $\CO\otimes_{R}R_v$ is the maximal order of $F_v$.

Let $d\in\N$ be such that $d+1$ is a prime number and let $D$ denote a  division algebra over $K$ \cite{Pierce} of dimension $(d+1)^2$. Let $D(R)$ denote a fixed maximal $R$-order in $D$.
Note that all maximal orders in $D$ are conjugate \cite{Reiner}.
For any $R$-algebra $A$ we define
$$
D(A)=D(R)\otimes_R A.
$$
Then $D(K)$ is canonically isomorphic to $D$.
For almost all places $v$, one has $D(K_v)\cong \M_d(K_v)$, where $\M_d(E)$ denotes the algebra of $d\times d$-matrices over a field $E$, \cite{Pierce}.
If $D(K_v)\cong \M_d(K_v)$, we say that $D$ \e{splits at $p$}.
If $D$ doesn't split at $p$, then $D(K_v)$ is a division algebra over $K_v$.
This latter fact rests on the choice of the degree $d$ to be a prime.

Let $S$ be the finite set of all places, at which $D$ doesn't split.

\begin{lemma}\label{lem1.1}
Let $A\subset D$ be a $K$-subalgebra.
Then the dimension of $A$ is 1,$d$ or $d^2$.
In the first case $A=K$, in the last $A=D$.
In the remaining case $A$ is a field extension of $K$ of degree $d$, such that every place $v\in S$ is non-decomposed in $A$, i.e., there is only one place of $A$ above $v$.
Every  field of degree $d$ over $K$, satisfying these conditions occurs as a subalgebra of $D$.
\end{lemma}

\begin{proof}
This lemma is standard. It can be pieced together from the information in Pierce's book \cite{Pierce}.
\end{proof}

Let $v$ be a place not in $S$.
Let $F/K$ be a field extension of degree $d$ which embeds into $D(K)$.
Then for any embedding $\sigma:F\hookrightarrow D(K)$ the set
$$
\CO_\sigma=\sigma^{-1}\(D(R)\)
$$
is an $R$-order in $F$.

\begin{lemma}\label{lem1.2}
Let $v$ be a place not in $S$.
Let $\sigma:F\hookrightarrow D(K)$ be a $K$-embedding of the degree $d$ field $F/K$.
Then for any $w\in S$, the order $\CO_{\sigma,w}=\CO_\sigma\otimes_{R}R_w$ is maximal in the local field $F_{w}$.
Conversely, let $\CO\subset F$ be an $R$-order such that for any $w\in S$ the order $\CO_{w}$ is maximal, then there exists an embedding $\sigma$ such that $\CO=\CO_\sigma$.
\end{lemma}

\begin{proof}
Analogous to the proof of Lemma 2.2 of \cite{class}.
\end{proof}

For a given degree $d$  field extension $F/K$ which embeds into $D(K)$, let $S_\inert(F)$ be the set of all $w\in S$ which are inert in $F$.
Define the \e{$S$-inertia degree} by
$$
f_S(F)=\prod_{w\in S}f_w(F)=d^{|S_\inert(F)|},
$$
where $f_w(F)$ is the inertia degree of $w$ in $F$.
For any order $\CO$ of $F$ let
$$
f_S(\CO)=f_S(F).
$$
Let $\CO$ be a $R$-order in $F$, which is maximal at all $w\in S$.
By Lemma \ref{lem1.2} there exists an embedding $\sigma:F\to D(K )$ such that $\CO=\CO_\sigma$.
Let $u\in D(R)^\times$ and let $^u\sigma$ be the embedding given by $^u\sigma(x)=u\sigma(x)u^{-1}$.
Then $\CO_{^u\sigma}=\CO_\sigma$, so the group $D(R)^\times$ acts on the set $\Sigma(\CO)$ of all $\sigma$ with $\CO_\sigma=\CO$.

\begin{lemma}\label{lem1.3}
The quotient $\Sigma(\CO)/D(R)^\times$ is finite and has cardinality 
$$
\left|\Sigma(\CO)/D(R)^\times\right|=f_S(\CO)h(\CO).
$$
\end{lemma}

Compare Lemma 2.3 in \cite{class}.

\begin{proof}
Fix an embedding $F\hookrightarrow D(K )$ and consider $F$ as a subfield of $D(K )$ such that $\CO=F\cap D(R)$. 
For $u\in D(K )^\times$ let
$$
\CO_u= F\cap u^{-1} D(R)u.
$$
Let $U$ be the set of all $u\in D(K )^\times$ such that $\CO_u=\CO$, i.e.,
$$
F\cap D(R)=F\cap u^{-1}D(R)u.
$$
Then $F^\times$ acts on $U$ by multiplication from the right and $D(R)^\times$ acts by multiplication from the left.
One has
$$
|D(R)^\times\bs U/F^\times|=|D(R)^\times\bs \Sigma(\CO)|.
$$
So we have to show that the left hand side equals $f_S(\CO)h(\CO)$.
For $u\in U$ let
$$
I_u=F\cap D(R)u.
$$
Then $I_u$ is a finitely generated $\CO$-module in $F$.
We claim that the map
\begin{align*}
\psi:D(R)^\times\bs U/F^\times &\to I(\CO)/F^\times,\\
u&\mapsto I_u
\end{align*}
is surjective and $h(\CO)$ to one.
We show this by localization and strong approximation.
For any place $w\ne v$ let $U_w$ be the set of all $u_w\in D(K_w)$ such that $\CO_w=F_w\cap D(R_w)= F_w\cap u_w^{-1}D(R_w)u_w$.
We have to show the following:
\begin{enumerate}[\rm (a)]
\item For $w\notin S$, the localized map $\psi_w:D(R_w)^\times\bs U_w/F_w^\times\to I(\CO_w)/F_w^\times$ is injective,
\item for $w\in S$, the map $\psi_w$ is $f_w(F)$ to one,
\item the map $\psi$ is surjective.
\end{enumerate}
For (a) let $w\notin S$, $u_w,v_w\in U_w$ and assume
$$
F_w\cap D(R_w)u_w=F_w \cap D(R_w)v_w.
$$
Let $z_w=v_wu_w^{-1}$.
Elementary divisor theory implies that there exist
$x,y\in D(R_w)^\times=\M_d(R_w)^\times$ such that
$$
z_w=x\mathrm{diag}(\pi^{k_1},\pi^{k_2})y
$$
holds, where $k_1\le k_2$ and $\pi$ is a uniformizer for $w$.
Replacing $u_w$ by $yu_w$ and $v_w$ by $x^{-1} v_w$ we may assume that $z_w$ equals the diagonal matrix.
The assumptions then  imply $k_1=0=k_2$, which gives the first claim.
For (b) let $w\in S$ and recall that $F_w$ is a local field, so $h(\CO_w)=1$.
Hence the claim is equivalent to
$$
|D(R_w)^\times\bs D(K _w)^\times/F_w^\times|=f_w(F).
$$
Taking the valuation $w$ of the reduced norm, one sees that the left hand side equals $d$ if $F_w$ is unramified over $K _w$ and $1$ otherwise, i.e., it equals the inertia degree $f_w(F)$ as claimed.

Finally, for the surjectivity of $\psi$ let $I\subset\CO$ be an ideal.
We show that there is $u\in D(K )^\times$ such that
$$
F\cap u^{-1} D(R)u= F\cap D(R)
$$
and 
$$
I=I_u=F\cap D(R)u.
$$
We do this locally.
First note that, since $I$ is finitely generated, there is a finite set $T$ of places with $T\cap S=\emptyset$ and $v\notin T$ such that for any $w\notin T\cup S$ the completion $I_w$ equals $\CO_w$ which is the maximal order of $F_w$.
For these $w$ set $\tilde u_w=1$.
Next let $w\in S$ and write $v_w$ for the unique place of $F$ over $w$.
Then $\CO_w$ is maximal, so is the valuation ring to $v_w$ and $I_w=\pi_w^k\CO_w$ for some $k\ge 0$, where $\pi_w$ is a uniformizer at $v_w$.
In this case set $\tilde u_w=\pi_w^k$.

Next let $w\in T$. Then $D(R_w)=\M_d(R_w)$. Let $\ol{\CO_w}=\CO_w/\pi_w\CO_w$ and $\ol{I_w}=I_w/\pi_wI_w$.
Then $\ol{\CO_w}$ is a commutative algebra over the field $\F_w=R_w/\pi_wR_w$, which implies that $\ol{\CO_w}\cong\bigoplus_{i=1}^s F_i$, where each $F_i$ is a finite field extension of $\F_w$.
Let $n_i$ be its degree.
Then there is an embedding $\ol{\CO_w}\hookrightarrow \M_d(\F_w)$ whose image lies in $\M_{n_1}(\F_w)\times\dots\times\M_{n_s}(\F_w)$. 
By the Skolem-Noether Theorem there is a matrix $\ol S\in \GL_d(\F_w)$ such that $\ol S\,\ol{\CO_w}\,\ol S^{-1}\subset \M_{n_1}(\F_w)\times\dots\times\M_{n_s}(\F_w)$.
The $\ol{\CO_w}$-ideal $\ol{I_w}$ must be of the form
$$
\ol{I_w}=\bigoplus_{i=1}^s\eps_i F_i,
$$
where $\eps_i\in\{ 0,1\}$.
Let $S$ be a matrix in $\GL_d(R_w)$ which reduces to $\ol S$ modulo $\pi_w$ and let $\tilde u_w=S^{-1}(l^{\eps_1}\Id_{n_1}\times\dots\times l^{n_s}\Id_{n_s})S$ in $\M_d(R_w)$.
By abuse of notation we also write $\tilde u_w$ for its reduction modulo $l$.
Then we have
$$
\ol{I_w}=\ol{\CO_w}\cap\M_d(\F_w)\tilde u_w.
$$
Let 
$$
I_{\tilde u_w}= F\cap D(R_w)\tilde u_w.
$$
Then it follows that
$$
\ol{I_w}\cong \ol{I_{\tilde u_w}}=I_{\tilde u_w}/\pi_wI_{\tilde u_w}
$$
and by Theorem 18.6 of \cite{Pierce} we get that $I_w\cong I_{\tilde u_w}$, which implies that there is some $\la\in F_w$ with $I_w=I_{\tilde u_w}\la$.
Replacing $\tilde u_w$ by $\tilde u_w\la$ and setting $\tilde u=(\tilde u_w)_w\in D(\A_\fin)$ we get
$$
I=F\cap D(R)\tilde u.
$$
By strong approximation there is an element $u\in D(K )^\times$ such that $D(\hat R)u=D(\hat R)\tilde u$ and therefore $I=I_u$.
\end{proof}

\section{Class numbers}

For any ring $R$ we write $\det:D(R)\to R$ for the reduced norm.
Note that this convention is compatible with the determinant, as for every field $F$, over which $D$ splits, the reduced norm equals the determinant.
We want to construct a group scheme $\CG$ over $R$ such that $\CG(F)=D(F)^\times/F^\times$ holds for every field.
Note that $D(R)$ is a free $R$-module of rank $d^2$.
Let $v_1,\dots v_{d^2}$ be a basis and note that the reduced norm
$\det(X_1v_1+\dots+X_{d^2}v_{d^2})$ is a homogeneous polynomial of degree $2$ in the variables $X_1,\dots,X_{d^2}$.
The group scheme $D^\times$ is given by the coordinate ring
$$
\CO_{D^\times}=R[X_1,\dots,X_{d^2},Y]/\(\det(X_1v_1+\dots+X_{d^2}v_{d^2})Y-1\).
$$
Now $\GL_1$ acts on $\CO_{D^\times}$ by
$$
\al f(X_1,\dots,X_{d^2},Y)=f(\al X_1,\al X_2,\dots,\al X_{d^2},\al^{-2}Y)
$$
and the coordinate ring we need is the ring of invariants
$$
\CO_\CG=\(\CO_{D^\times}\)^{\GL_1},
$$
which is the subring generated by the elements $X_iX_iY$ with $1\le i\le j\le d^2$.

\begin{lemma}
The ring $\CO_\CG$ is the coordinate ring of an affine group scheme $\CG$ over $R$ such that for every factorial ring $R'/R$ one has
$$
\CG(R')= D(R')^\times/R'^\times.
$$
\end{lemma}

\begin{proof}
The first claim is clear. We prove the second first in case of a field.
Consider the exact sequence of group schemes
$$
1\to \GL_1\to D^\times\to\CG\to 1.
$$
For any field $K$ this gives an exact sequence of groups
$$
1\to \GL_1(K)\to D(K)^\times\to\CG(K)\to H^1(K,\GL_1),
$$
where the last item is the Galois-cohomology, which vanishes by Hilbert's Theorem 90.
This implies the claim for fields.
Now let $R$ be a factorial ring, so $R$ is integral and has unique factorization.
Write $K$ for its quotient field and let $\chi\in \CG(R)$, so $\chi$ is a ring homomorphism from $\CO_\CG$ to $R$.
By the first part of the proof, $\chi$ extends to a ring homomorphism $\tilde\chi:\CO_{D^\times}\to K$.
We show that this lift can be modified so as to have values in $R$.
For $1\le i\le d^2$ we have
$$
\tilde\chi(x)^2\tilde\chi(y)=\chi(x^2y)\in R.
$$
If $p$ is an irreducible in $R$ which divides the denominator of, say,  $\chi(x_1)$, we replace any $\tilde\chi(x_i)$ by $p\tilde\chi(x_i)$ and $\tilde\chi(y)$ by $\frac1{p^2}\tilde\chi(y)$ without changing $\chi$, so we can assume $\tilde\chi(x_1)$ to lie in $R$.
We repeat this with $x_2$ and so on.
Now if $\tilde\chi(y)$ does not lie in $R$ there must be an irreducible $p$ dividing its denominator. But as the product is in $R$, $p$ also divides $\tilde\chi(x_i)^2$ hence $\tilde\chi(x_i)$.
So we can replace $\tilde\chi(x_i)$ by $\frac1p\tilde\chi(x_i)$ and $\tilde\chi(y)$ by $p^2\tilde\chi(x_i)$ and by repeating this procedure we arrive at $\tilde\chi(x_i)$ and $\tilde\chi(y)$ both lying in $R$.
\end{proof}

We set $\Ga=\CG(R)$.
By Theorem 3.2.4 in \cite{Margulis}, the group $\Ga$ is a uniform lattice in $G=\CG(K_v)\cong \PGL_2(K_v)$, i.e., $\Ga$ is a discrete subgroup of $G$ such that $\Ga\bs G$ is compact.

Let $v_1,\dots,v_d$ denote the standard basis of $\R^d$.

\begin{theorem}
Let $S$ be a set of places of $K$ with $|S|\ge 2$.
Let $\CF(S)$ denote the set of field extensions $F/K$ such that every place $v\in S$ is non-decomposed in $F$.
Let $O(S)$ denote the set of all orders $\CO\subset F$ where $F\in\CF(S)$ such that $\CO$ is maximal at every $v\in S$.
For $\CO\in O(S)$ let $R(\CO)$ be its regulator.
Then the sum
$$
N(k)=\sum_{\substack{\CO\in O(S)\\ \ga\in\CO^\times\\ \la(\a_\ga)=k}}R(\CO)h(\CO) f_S(\CO),\qquad k\in\N^d,
$$
satisfies
$$
N(k)\sim\1_\La(k)\,C_{\La/\La'}(k)\, c^k
$$
for $k_j\to\infty$ for every $j=1,\dots,d$ independently.
Here we have
$$
c_j=\(\prod_{\nu=0}^{j-1}\frac{q^{d+1}-q^\nu}{q^j-q^\nu}\)^{d+1},\qquad j=1,\dots,d,
$$
where $q$ is the residue cardinality of $R_v$.
Further $C_{\La/\La'}:\La/\La'\to (0,\infty)$.
The lattice $\La\subset\R^d$ is generated by
$$
f_0=\frac2d(v_1+\dots+v_d),\quad f_j=f_0-2\frac{d+1}d v_j,
$$
where $j=1,\dots,d$ and $\La'\subset\La$ is the sublattice generated by $2v_1,\dots,2v_d$.
\end{theorem}

\begin{proof}
The group $G_\ga$ equals $F_\ga^\times$, where $F_\ga$ is the field extension given by the centralizer of $\ga$. Therfore one sees that $|G_\ga^\a/\Ga_\ga^\a|$ equals the regulator of the order $\CO_\ga$. 
The theorem now follows from the prime geadesic theorem together with Lemma \ref{lem1.2} and Lemma \ref{lem1.3}.

The computation of the $c_j$ follows their definition in \cite{BL}, they are given as $c_j=|Ka_jK/K|$. For this computation one can assume that $K=\SL_{d+1}(R_v)$.
Let $m_j$ be the diagonal matrix with entries $(\pi^{-1},\dots,\pi^{-1},1,\dots,1)$ with $j$-times $\pi^{-1}$, where $\pi$ is a uniformizer of the discrete valuation ring $R_v$.
Then two elements $b,b'$ of $Km_jK$ lie in the same $K$-coset if and only if the $R_v$-span of $b e_1,\dots,be_{d+1}$ equals the $R_v$-span of $b' e_1,\dots,b'e_{d+1}$, where $e_1,\dots,e_{d+1}$ is the standard basis of $K_v^{d+1}$.
Taking this modulo the $R_v$-submodule spanned by $e_1,\dots,e_{d+1}$ we end up determining the number of $j$-dimensional $\F_q$ sub vector spaces of $\F_q^{d+1}$, which is $\prod_{\nu=0}^{j-1}\frac{q^{d+1}-q^\nu}{q^j-q^\nu}$ (The enumerator goves the number of bases and the denominator the number of invertible matrices for a given bases.) This number is $|Km_jK/K|$.
But now $a_j=m_j^{d+1}$ and the map $a\mapsto |KaK/K|$ is a quasi-character on $A^-$, see Lemma 2.4.5 of \cite{BL}.
This concludes the computation of $c_j$ and finishes the proof of the theorem.
\end{proof}

{\small Mathematisches Institut\\
Auf der Morgenstelle 10\\
72076 T\"ubingen\\
Germany\\
\tt deitmar@uni-tuebingen.de}

\today

\end{document}